\documentclass[12pt,a4paper,twoside,english]{smfart}

\usepackage[english]{babel}
\usepackage{lmodern,mdframed,xcolor}
\usepackage[utf8]{inputenc}

\usepackage[OT2,T1]{fontenc}

\usepackage{amssymb}
\usepackage{amsmath}
\usepackage{smfthm}
\usepackage{mathabx}
\usepackage{enumerate}
\usepackage{amsthm}
\usepackage{amsfonts}
\usepackage{graphicx}
\usepackage[colorlinks=true,linkcolor=black,citecolor=black,urlcolor=black]{hyperref}
\usepackage{fancyhdr}
\usepackage[all,cmtip]{xy}
\usepackage{alltt}
\usepackage{footmisc}

\usepackage{calrsfs}

\usepackage{fullpage}

\def\Q{{\mathbb Q}}
\def\Z{{\mathbb Z}}
\def\fq{{\mathbb F}}
\def\Com{{\mathbb C}}

\def\p{{\mathfrak p}}

\def\Gg{{\mathfrak g}}

\def\sl{{\mathfrak s} {\mathfrak l}}
\def\gl{{\mathfrak g} {\mathfrak l}}

\def\LL{{\mathfrak L}}

\def\GG{{\mathcal G}}
\def\O{{\mathcal O}}
\def\U{{\mathcal U}}
\def\E{{\mathcal E}}

\def\ff{{\mathcal F}}

\def\Ll{{\mathcal L}}

\def\T{{\mathcal T}}

\def\Reg{{\rm Reg}}
\def\Ind{{\rm Ind}}
\def\1{{\bf 1}}

\newcommand{\isomto}{\overset{\sim}{\rightarrow}}

\newtheorem{Theorem}{Theorem}

\newtheorem*{Remark}{Remark}

\newtheorem*{Example}{Example}

\author{Christian Maire}
 \address{FEMTO-ST Institute, Universit\'e Bourgogne Franche-Comt\'e, CNRS,  15B avenue des Montboucons, 25000 Besan\c con, FRANCE} 
\email{christian.maire@univ-fcomte.fr}

\begin{document}

\date{\today}

\title{On Galois representations  with large image}

\begin{abstract} For  every   prime number $p\geq 3$ and every integer $m\geq 1$, we prove the existence of a continuous Galois representation~$\rho: G_\Q \rightarrow Gl_m(\Z_p)$  which has open image and  is unramified outside $\{p,\infty\}$ (resp. outside $\{2,p,\infty\}$)  when $p\equiv 3$ mod $4$ (resp.  $p \equiv 1$ mod $4$). 
\end{abstract}

\thanks{The author thanks Anwesh Ray for many useful and inspiring discussions,  Ravi Ramakrishna for  interesting exchanges,  Farshid Hajir for his strong encouragement and comments, and C\'ecile Armana for useful remarks. This work  was partially supported by the ANR project FLAIR (ANR-17-CE40-0012) and  by the EIPHI Graduate School (ANR-17-EURE-0002).  }

\subjclass{11F80, 11R37, 11R32}

\keywords{Galois representations, uniform groups, pro-$p$ extensions unramified outside $p$.}


\maketitle


Let $K$ be a number field having $r_2$ non-real embeddings, let $p$ be a prime number and let~$G$ be a finitely generated pro-$p$ group of $p$-rank at most  $r_2+1$. When the field~$K$ is $p$-rational
(see \S \ref{section_restricted_ramification} for the full definition and background), the Galois group of the maximal  $p$-extension of $K$ unramified outside $p$ is a free pro-$p$ group of rank $r_2+1$. Hence 
the group  $G$  can be realized as the  Galois group of an extension over~$K$ unramified outside $p$, thanks to the universal property of free groups. 
In the context of Galois representations, Greenberg in \cite{Greenberg} developed this approach to realize continuous Galois representations $\rho:G_\Q\rightarrow Gl_m(\Z_p)$ of the absolute Galois $G_\Q$ of  $\Q$, with open image and such that $\rho$ is  unramified outside $\{p,\infty\}$, under the hypotheses that $p$ is a regular prime and $m$ satisfies $ 1+ 4[m/2] \leq p$. 
The regularity of $p$ is important because for the cyclotomic field $K=\Q(\zeta_p)$, it is equivalent to the $p$-rationality of $K$.

\medskip

A few years later this method was  extended by  Cornut and J. Ray \cite{Cornut-Ray} for more general linear groups, but always under the assumption that $p$ is regular and that all large $m$ are excluded when $p$ is fixed.

\medskip

In fact, it is possible to relax the condition on  $p$-rationality to realize Galois representations with big image:
this has been recently done by A. Ray  in \cite{Ray}. For example, when $p\geq 2^{m+2+2e_p}$, where $e_p$ is the index of irregularity of~$p$, A. Ray shows the existence of continuous Galois representations $\rho: G_\Q \rightarrow Gl_m(\Z_p)$ unramified outside $\{p,\infty\}$ with open image. But as in \cite{Greenberg} and \cite{Cornut-Ray}, the dimension of the representations is bounded for fixed~$p$. 

\medskip

By a different approach, Katz in \cite{Katz} constructs geometric Galois representations over cyclotomic extensions, and by descent he gets finitely ramified continuous Galois representations of  $G_\Q$  with open image in $Gl_m(\Z_p)$,   for $p\equiv 1 $ mod $3$ or $p \equiv 1$ mod~$4$ for every even $m\geq 6$. In particular, for such primes $p$, the result of Katz shows the existence of Galois representations with open image for large~$m$. We note that the representations constructed by Katz are motivic but are ramified at sets consisting of primes of potentially many different residue characteristics, whereas the earlier approach yields representations 
unramified outside $\{p,\infty\}$ which are, by contrast, what Katz calls ``spectacularly non-motivic''.

\medskip

In this work, by extending the arithmetical approaches of \cite{Greenberg}, \cite{Cornut-Ray} and \cite{Ray}, we are able to prove (Corollary \ref{coro_maintheorem}):

\begin{Theorem}\label{TheoremA} Given a prime number $p\geq 3$, and an integer $m\geq 1$, there exist continuous Galois representations $\rho: G_\Q \rightarrow Gl_m(\Z_p)$ with open image satisfying:
   \begin{itemize}
       \item[$(i)$] $\rho$ is unramified ouside $\{p, \infty\}$ if $p\equiv -1 \ {\rm mod} \ 4$,
       \item[$(ii)$] $\rho$ is unramified ouside $\{2,p, \infty\}$ if $p\equiv 1 \ {\rm mod} \ 4$.
   \end{itemize}
\end{Theorem}

\medskip

 \begin{Remark}For $m=1$ the existence of such representations is a consequence of class field theory. \end{Remark}

Our criteria coincide with those  of  Greenberg when the  number field $K$ fixed  by the residual representation is $p$-rational. 
However our approach also works in greater generality. In particular by passing through the number field  $K=\Q(\zeta_p)$, the criteria we give are specially adapted to produce,  for many primes  $p\equiv 1$ mod~$4$  and large $m$, continuous Galois representations $\rho: G_\Q \rightarrow Gl_m(\Z_p)$ ramified only at $\{p,\infty\}$ with open image. 
In fact, this is the case for all but six primes $p\equiv 1$ mod~$4$ less than $ 4\cdot 10^5$.
The main technical result we obtain can be viewed as a refinement of Theorem \ref{TheoremA}, $(ii)$.
To explain it, let $v_2$ be the $2$-adic valuation, and let $\omega$ be the mod $p$ reduction of the cyclotomic character. We prove (Theorem \ref{theo_zeta_p}):

\begin{Theorem} Let $p \equiv 1 \ {\rm mod \ } 4$ be a prime number, and let $m\geq 3$. 
Write $p-1=2^\lambda a$ where $2\nmid a$, so $\lambda=v_2(p-1)$. Let $\{\omega^{k_1},\cdots, \omega^{k_e}\}$ be the characters corresponding to the nontrivial components of  the $p$-Sylow of the  class group of $\Q(\zeta_p)$.
Suppose that:
\begin{itemize}
 \item[$(i)$] $v_2(m-1) \geq \lambda$ if $m$ is odd and $v_2(m-2)\geq \lambda$ if $m$ is even;
  \item[$(ii)$] $a\nmid (k_i-1)$ for $i=1,\cdots, e$.
\end{itemize} Then there exist continuous Galois representations $\rho: G_\Q \rightarrow Gl_m(\Z_p)$  unramified outside $\{p,\infty\}$, and  with open image.
\end{Theorem}

\medskip

\begin{Example} 
For $ p\leq 4\cdot 10^5$, there are only six cases for which $(ii)$ fails, and the index of irregularity $e$ is 1 for all of them:
$$\begin{array}{l|cccccc}
 p & 257&3329 &11777 &114689&163841&184577\\
 \hline
 k_1 &93& 1951&8879&34343&140801&49029
\end{array}
\cdot $$
\end{Example}

\medskip

Here is a sketch of our approach. We first revisit  the question of  the lifting of residual Galois  representations (of order coprime to $p$) in terms of embedding problems, by using the criteria of Hoechsmann (see for example \cite[Chapter III, \S 5]{NSW}). The result we obtain involves the adjoint representation of a uniform group~$G$ (Theorem \ref{theo_main}).
We then exploit a result of Kuranishi \cite{Kuranishi}  that shows that a semisimple Lie algebra  can be generated by~$2$ elements; in particular we use the explicit form  for  $\sl_m$ recently given by Detinko-De Graaf \cite{Detinko-DeGraaf}, and Chistopolskaya \cite{chisto-sln}. 
Thus we apply our embedding  criteria to some special subgroup $G'$ of $Sl_m(\Z_p)$  generated by two elements. 
And instead of considering number fields of large degree, namely $\Q(\zeta_p)$, we reduce the study of the existence of Galois representations with open image, to properties of certain imaginary quadratic extensions. 

\medskip

In this work we restrict our attention to the problem introduced by Greenberg  \cite{Greenberg} for the group $Sl_m(\Z_p)$. But it seems likely the methods we introduce will apply more generally for realizing other groups as well.

\medskip

The paper contains 4 sections. In Section 2 and in Section 3,  we recall facts from the maximal pro-$p$-extension of a number field unramified outside $p$, then  generalities regarding uniform groups and  $\Z_p$-Lie algebras.
 In Section 4,   we develop the approach of lifting mod $p^k$ representations as a question of  embedding problem; in particular we give criteria for lifting in some given uniform group (Theorem \ref{theo_main}).
 The last section is devoted to applications; in particular we prove the results presented in the Introduction.

\

{\bf Notations.} 
Throughout this article $p$ is a prime number.

$\bullet$ If $M$ is a finitely generated $\Z_p$-module, set $d_p M:=dim_{\fq_p} M/M^p$, $M[p]:=\{m\in M, m^p=1\}$, and $Tor(M)=\{m \in M, \exists k, m^{p^k}=1\}$.

$\bullet$ If $G$ is a pro-$p$ group, set $G^{ab}:=G/[G,G]$, $G^{p,el}:=G^{ab}/(G^{ab})^p$, and $d_p G:=d_p G^{ab}$.

$\bullet$ If $A$ is a Hausdorff, abelian and locally compact topological group, set $A^\wedge$ the Pontryagin dual of $A$.

\smallskip

For the computations  we have used the program PARI/GP \cite{pari}.

\smallskip

\section{On the maximal pro-$p$ extension unramified outside $p$: the results we need}

\subsection{On pro-$p$ groups} For classical properties on cohomology and homology of pro-$p$ groups, see for example \cite[Chapters I and II]{NSW}. 

Let $1 \longrightarrow G \longrightarrow \Gamma \longrightarrow \Delta \longrightarrow 1$ be an exact sequence of profinite groups where $G$ is a finitely presented pro-$p$ group, and $\Delta$ is a finite group of order coprime to $p$.
Recall that by the Schur-Zassenhaus Theorem one has  $\Gamma \simeq G \ltimes \Delta$.

\begin{prop} \label{prop_H2}
 Let $M$ be a finite $\Gamma$-module of exponent $p$ on which $G$ acts trivially. Then for $i\geq 1$, we have the isomorphism:
 $H^i(\Gamma,M) \simeq (H^i(G,\Z/p)\otimes M)^\Delta.$
\end{prop}

\begin{proof}
 First, by the algebraic universal coefficients Theorem for $G$-homology  over $\fq_p$, one has the isomorphism
 \begin{equation}\displaystyle{\label{equation1} F :  H_i(G,\Z/p)\otimes {M}^\wedge \ {\isomto}  \ H_i(G, {M}^\wedge),}\end{equation}
 where the tensor product is taken over $\fq_p$, and where $F$ is defined by
 $$F([f]\otimes m)=[f\otimes m],$$ showing that (\ref{equation1}) is also an isomorphism of $\Delta$-modules. See for example \cite[Chapter VI, \S 15, Theorem 15.1]{Hilton-Stammbach}.
 By Pontryagin duality, we obtain $H^i(G,M) \simeq H^i(G,\Z/p)\otimes M$, as $\Delta$-modules.
 Since $|\Delta|$ is coprime to $p$, by the Hochschild-Serre spectral sequence one also has $H^i(\Gamma,M)\simeq H^i(G,M)^\Delta$ (see for example  \cite[Chapter II, \S 1, Lemma 2.1.2]{NSW}).
 By combining these two observations we finally obtain the claimed isomorphism.
\end{proof}

Let us write $G^{ab} \simeq \Z_p^t \oplus \T$, where $\T$ is the torsion subgroup of $G^{ab}$.

\begin{prop} \label{prop_H2_Tor}  
 Let $M$ be a finite $\Gamma$-module of exponent $p$ on which $G$ acts trivially. 
If $H^2(G,\Q_p/\Z_p)=0$ then  $H^2(G,M) \simeq \big({\T[p]}^\wedge\otimes M\big)^\Delta$.
\end{prop}

\begin{proof}
 By taking the $G$-homology of the exact sequence $0 \longrightarrow \Z_p \longrightarrow \Z_p \longrightarrow \Z/p\Z \longrightarrow 0$, we get the exact sequence of $\fq_p[\Delta]$-modules
$$\xymatrix{ H_2(G,\Z_p)/p \ar@{->}[r] & H_2(G,\Z/p\Z) \ar@{->>}[r]& H_1(G,\Z_p)[p] .}$$ 
After observing that $H_2(G,\Z_p)^\wedge \simeq H^2(G,\Q_p/\Z_p)=0$, then $H^2(G,\Z/p)$ is isomorphic to  $\big(H_1(G,\Z_p)[p]\big)^\wedge \simeq \T[p]^\wedge$, and we conclude with  Proposition~\ref{prop_H2}.
\end{proof}
By the way, the  proof of Proposition \ref{prop_H2_Tor} allows us to obtain:
\begin{prop} \label{prop2}
 One has $$d_pH^1(G,\Z/p\Z)-d_pH^2(G,\Z/p\Z)=t-d_pH_2(G,\Z_p).$$
\end{prop}


\subsection{Restricted ramification}  \label{section_restricted_ramification}
Let $K$ be a number field. To simplify when $p=2$ we assume $K$ totally imaginary.
Set
\begin{itemize}
    \item[$\bullet$] $E_K:=\Z_p\otimes \O_K^\times$ the pro-$p$ completion of the  group of units of  the ring of integers $\O_K$ of $K$,
     \item[$\bullet$] $Cl_K$ the $p$-Sylow of the class group of $K$,
   \item[$\bullet$]  $K_\p$ the completion of $K$ at $\p|p$, $U_\p$ the local units of $K_\p$,
\item[$\bullet$]    $\displaystyle{\U_\p:= \lim_{\stackrel{\longleftarrow}{n}} U_\p/U_\p^{p^n}}$  the pro-$p$ completion of  $U_\p$, and $\displaystyle{\U_p:=\prod_{\p|p}\U_\p}$,
\item[$\bullet$] $\iota_{K,p} : E_K \rightarrow \U_p$  the diagonal embedding of $E_K$ into $p$-adic units.
\end{itemize}

\subsubsection{The pro-$p$ group $G_{K,p}$}

Let $K_p/K$ be the maximal pro-$p$ extension of $K$ unramified outside $p$; set $G_{K,p}=Gal(K_p/K)$.
The pro-$p$ group $G_{K,p}$ is finitely presented. More precisely, one has (see \cite[Chapter VIII, Proposition 8.3.18; Chapter X, Corollary 10.4.9, Theorem 10.7.13]{NSW}):

\begin{theo} \label{theo_GS}
 The pro-$p$ group $G_{K,p}$ is of cohomological dimension $1$ or $2$, and
 $d_p H^1(G_{K,p},\Z/p) - d_pH^2(G_{K,p},\Z/p)=r_2+1$.
\end{theo}
Here as usual $(r_1,r_2)$ is the signature of $K$.

Let us write $G_{K,p}^{ab}\simeq \ff_{K,p}\oplus \T_{K,p}$, where $\T_{K,p}:=Tor(G_{K,p}^{ab})$ is the torsion of $G_{K,p}^{ab}$, and where   $\ff_{K,p}:=G_{K,p}^{ab}/\T_{K,p} \simeq  \Z_p^{t_p}$ is the free part; the quantity  $t_p$ is the $\Z_p$-rank of $G_{K,p}^{ab}$.
By class field theory one has:
\begin{eqnarray} \label{se1} t_p=dim_{\Q_p} coker (\iota_{K,p})=r_2+1 + dim_{\Q_p} ker(\iota_{K,p}).
 \end{eqnarray} 
 (See for example \cite[Chapter III, \S 1, Corollary 1.6.3]{gras}.)
 
 Recall also that Leopoldt's conjecture asserts that $ker(\iota_{K,p})=1$, and 
 thanks to Baker and Brumer \cite{Brumer} one knowns that Leopoldt's conjecture is true for abelian extensions $K/\Q$. 
One also has the following well-known result (see for example \cite[Chapter X, Corollary 10.3.7]{NSW}):
\begin{prop} \label{prop_injectivity_iota}One has
  $ker(\iota_{K,p})=1 \Longleftrightarrow H_2(G_{K,p},\Z_p)=1$.
\end{prop}

\begin{proof}
By Proposition  \ref{prop2} and Theorem \ref{theo_GS} one has:
$$t_p- d_pH_2(G_{K,p},\Z_p)= r_2+1;$$
thus by combining with 
 $(\ref{se1})$, we get:
$dim_{\Q_p} ker(\iota_{K,p})= d_pH_2(G_{K,p},\Z_p)$.
 Observe now that $H_2(G_{K,p},\Z_p)$ is an abelian pro-$p$ group,  then $H_2(G_{K,p},\Z_p)$ is trivial if and only if $d_p H_2(G_{K,p},\Z_p)=0$.
\end{proof}


Regarding $\T_{K,p}$, we have the following: 

\begin{prop}\label{prop_tor}
 Suppose $Cl_K=1$. Then $\T_{K,p} \simeq Tor\Big(\U_p/\iota_{K,p}(E_K)\Big)$.
\end{prop}

\begin{proof}
By class field theory one has $\U_p/\iota_{K,p}(E_K) \simeq G_{K,p}^{ab}$ when $Cl_K=1$.
\end{proof}

Hence, given a number field $K$, up to a finite set of primes (those that divide $|Cl_K|$) the computation of $\T_{K,p}$ is reduced to the computation of the torsion of $\U_p/\iota_{K,p}(E_K)$. 
And having some nontrivial element in $Tor\big(\U_p/\iota_{K,p}(E_K)\big)$ is something that is rare; typically
one has the following conjecture (\cite[Conjecture 8.11]{Gras-CJM}).

\begin{conj}[Gras] \label{conjecture_gras} Given a number field $K$, then $\T_{K,p}=1$ for $p\gg 0$. 
\end{conj}
Regarding this conjecture many computations allow us to have some evidence, but very little is known in general.
See \cite[Chapter IV, \S 3 and \S 4]{gras} and \cite{Gras-CMS} for a good exposition.  
 Nevertheless, the $p$-group $\T_{K,p}$ is a deep arithmetical object associated to $K$, as we can see from the following result, for example.
  
  \begin{prop}\label{prop_free}  The pro-$p$ group $G_{K,p}$ is free pro-$p$ (on $r_2+1$ generators) if and only if $ker(\iota_{K,p})=1$ and $\T_{K,p}=1$.
   \end{prop}

\begin{proof} If $G_{K,p}$ is free pro-$p$  then $G_{K,p}^{ab}\simeq \Z_p^{t_p}$, $\T_{K,p}=1$, $H^2(G_{K,p},\Q_p/\Z_p)=0$, and by Proposition \ref{prop_injectivity_iota} one gets $ker(\iota_{K,p})=1$.

For the reverse, suppose that $ker(\iota_{K,p})=1$ and $G_{K,p}\simeq \Z_p^{t_p}$. By Proposition \ref{prop_injectivity_iota}, $H_2(G_{K,p},\Z_p)=0$;  by Proposition \ref{prop_H2_Tor}, one gets $H^2(G_{K,p},\Z/p)=0$ (take $\Delta$ trivial and $M=\Z/p$), and then $G_{K,p}$ is pro-$p$ free.

Regarding the $p$-rank of $G_{K,p}$, see Theorem \ref{theo_GS}.
\end{proof}

\begin{exem} \label{exemple_quad_imaginary}
 Take $p>3$, and let $K/\Q$ be an imaginary quadratic field. Observe that  $E_K=1$ and that $\U_p$ is torsion free. Hence when $Cl_K=1$, the pro-$p$ group $G_{K,p}$ is free pro-$p$  on $2$ generators.
\end{exem}

Finally, let us recall that when $G_{K,p}$ is free pro-$p$ then $K$ is said to be {\it $p$-rational} (\cite{Movahhedi-phd}).

\subsubsection{With semisimple action} \label{section_semisimple}

Let $\Delta$ be a finite group of order coprime to~$p$.
Let $\Psi_p$ be the set of irreducible $\Com_p$-characters of $\Delta$. 
Let $M$ be a finite  $\fq_p\lbrack \Delta \rbrack$-module. For $\varphi \in \Psi_p$, set $r_\varphi M$
the $\varphi$-rank of $M$: that is the number of times that $\varphi$  appears in the  decomposition of $M$ as $\fq_p\lbrack \Delta \rbrack$-module. In particular if $\chi(M)$ denotes the character of $M$, then $\chi(M)=\sum_{\varphi \in \Psi_p} r_\varphi \varphi$.
Observe that for a finite $\Z_p\lbrack \Delta \rbrack$-module $M$, one has $\chi(M/M^p)=\chi(M[p])$.

\begin{defi}
Two finite $\fq_p[\Delta]$-modules $M$ and $N$ are said to be {\it orthogonal}, and write $M\perp N$, if  for every $\varphi \in \Psi_p$ one has $r_\varphi M \cdot r_\varphi N=0$. 
\end{defi}

Since $\chi(M\otimes N)=\chi(M) \chi(N)$ and $\chi(M^\wedge)=\chi(M)^{-1}$, one has:
\begin{lemm}\label{lemm_character_trivial}
 Let $M$ and $N$ be two finite $\fq_p\lbrack \Delta \rbrack$-modules.
 
 Then $\Big(M^\wedge \otimes N\Big)^\Delta=0$ if and only if $M\perp N$. 
\end{lemm}

We denote by $\Reg$ the character of the regular representation,  by $\1$ the trivial character, and for a subgroup $D$ of $\Delta$, by $\Ind_D^\Delta\1_D$ the induced character from $D$ to $\Delta$ of the trivial character  $\1_D$ of $D$.

\medskip

For the end of this section, let us consider the following frame.

Let $K/k$ be a finite Galois extension of degree coprime to $p$; put $\Delta=Gal(K/k)$. Observe that $K_p/k$ is Galois and that $\Delta$ acts on $G_{K,p}$, $\T_{K,p}$, $\ff_{K,p}$, etc. 
Put $\Gamma=Gal(K_p/k)\simeq G_{K,p} \ltimes \Delta$.

As we will see, we need that the two pieces $\ff_{K,p}$ and $\T_{K,p}$ of  $G_{K,p}^{ab}$ must be 
orthogonal to each other (as $\Delta$-modules). 
First, the next Theorem will be essential to lift residual representation.

\begin{theo} \label{theo_H2_semisimple}  Let $M$ be a finite $ \Gamma $-module of exponent $p$ on which $G_{K,p}$ acts trivially. Assuming Leopoldt's conjecture for $K$ at $p$, then $H^2(\Gamma,M) \simeq \big({\T_{K,p}[p]}^\wedge\otimes M\big)^\Delta$.
 In particular  $H^2(\Gamma,M)=0$ if and only if $\T_{K,p}[p]\perp M$.
\end{theo}

\begin{proof}
 This is a consequence of Proposition \ref{prop_H2_Tor}, Proposition \ref{prop_injectivity_iota} and Lemma \ref{lemm_character_trivial}.
\end{proof}

\begin{rema} 
 When $K$ contains $\zeta_p$,   the character of  $\T_{K,p}[p]$ is related to the mirror character of  $Cl_K'$, where $Cl_K'$ is the $p$-Sylow of the $p$-class group of $K$.
 Typically when  $K=\Q(\zeta_p)$,     $r_\varphi \T_{K,p}[p]=r_{\varphi^*}Cl_K$, where $\varphi^*:=\omega \varphi^{-1}$.
  And in this case, $\Q(\zeta_p)$ is $p$-rational if and only if $p$ is regular. For more general results see \cite{gras2}.
\end{rema}

To finish, the following proposition will be the starting point for realizing residual representations as Galois extensions of number fields.

\begin{prop} \label{prop_character_Gp} Assuming the Leopoldt conjecture for $K$ at $p$, one has $$\chi(\ff_{K,p}/p)=\1+n\Reg -\sum_{v|\infty} \Ind_{D_v}^G \1_{D_v},$$
where $n=[k:\Q]$.
 In particular if $K/k$ is a CM-field one has  $\chi(\ff_{K,p}/p)=\1+n \varphi$, where $\varphi$ is the nontrivial character of $Gal(K/k)$.
\end{prop}

\begin{proof} One has $\Q_p\otimes \ff_{K,p}=\Q_p\otimes \U_p \Big/ \Q_p \otimes  \iota_{K,p}(E_K)$.
 Then use for example \cite[\S 5 Theorem 5.12, and \S 6]{gras2}.
\end{proof}


\section{Uniform groups and Lie algebras}

\subsection{Generalities}
For this section we refer to \cite[Chapters 4, 7 and 9]{DSMN}.

Let $G$ be a finitely generated pro-$p$ group. Set $G_1=G$, and for $n\geq 1$, $G_{n+1}=G^p_n[G,G_n]$. The  $(G_n)$ is the $p$-descending central series of $G$.  
  For $n\geq 1$, consider the morphism: $$ \begin{array}{rcl}
                 \alpha_n : G_n/G_{n+1} & \rightarrow& G_{n+1}/G_{n+2} \\
                 x &\mapsto & x^p.
                 \end{array}
  $$
  
 \begin{defi}  The pro-$p$ group $G$ is said to be {\it uniform} if  for every $n$, the map $\alpha_n$ is an isomorphism.
 \end{defi}
 
Hence when $G$ is uniform, there exists some $d$ such that $G_n/G_{n+1} \simeq (\Z/p)^d$; the integer~$d$ is called the   dimension of~$G$.

\begin{theo} \label{prop_uniform} Let $G$ be a uniform pro-$p$ group. Then for all $n\geq 1$, $G_{n+1}$ is uniform and  also equal to:
 \begin{itemize}
  \item[$(i)$] $G_n^p[G_n,G_n]$,
  \item[$(ii)$] $G^{p^{n}}=\langle g^{p^n} ,g \in G \rangle=\{g^{p^n}, g\in G\}$,
  \item[$(iii)$] $(G_n)^p$.
 \end{itemize}
\end{theo}

\begin{proof}
 See \cite[Chapter 3, Theorem 3.6]{DSMN}.
\end{proof}

Recall that a {\it $p$-adic analytic group} is a topological group  $G$ having a structure of $p$-adic analytic manifold for which the addition and the inverse are analytic.
Since Lazard \cite{Lazard} one knows that uniform pro-$p$ groups are the socle of $p$-adic analytic groups. 
Indeed:

\begin{theo}\label{theo_uniform}
 
 $(i)$ A  uniform  group $G$ of dimension $d$  is a $p$-adic analytic  group of dimension $d$ (as analytic manifold).
 
 $(ii)$ Every $p$-adic analytic group of (analytic) dimension $d$ contains an open subgroup which is uniform of dimension $d$.
 
 $(iii)$ Let $G$ be a pro-$p$ group which is a $p$-adic analytic group,  then $G\hookrightarrow Gl_m(\Z_p)$ for some $m$.
\end{theo}

\begin{proof}
See \cite[Interlude A]{DSMN}. 
\end{proof}

In what follows, we will  consider uniform groups $G$ as subgroups of $Gl_m(\Z_p)$.

\subsection{Exponential and logarithm}

\subsubsection{The Lie algebras $\gl_m$ and $\sl_m$} 
Set $\varepsilon=0$ if $p>2$, and $\varepsilon=1$ if $p=2$. 

\smallskip 

Take $m\geq 2$.
Let $\gl_m$ be the $\Z_p$-free module of dimension $m^2$  generated by the matrices $E_{i,j}(p):=p^{1+\varepsilon}E_{i,j}$, where $E_{i,j}$ are the elementary matrices.
Then $\gl_m$ is a  $\Z_p$-Lie algebra, subalgebra of the algebra $\gl_m(\Q_p)$ of the  matrices of size $m\times m$ with coefficients in $\Q_p$, equipped with the Lie bracket $(A,B)=AB-BA$.

It is not difficult to see that $(\gl_m,\gl_m) \subset p^{1+\varepsilon}\ \gl_m$: the algebra $\gl_m$ is said to be  {\it powerful}.

Thanks to \cite[Chapter IV, Theorem 1.3.5.1]{Lazard}, one knows that the exponential map $exp: x\mapsto \sum_{n\geq 0}(n!)^{-1} x$ and the logarithm map $log(z):=\sum_{n\geq 1} (-1)^{n+1}n^{-1} (z-1)^n$ converge for $x \in \gl_m$  and  $z\in Gl_m^{1}$, where $Gl_m^1=\{A \in Gl_m(\Z_p), A\equiv  1 \ {\rm mod \ } p^{1+\varepsilon}\}$. Moreover $exp$ and $log$  are reciprocal on these two spaces.
Hence $exp(\gl_m)=Gl_m^{1}$ and 
since $\gl_m$ is powerful, $Gl_m^{1}$ is uniform (\cite[Chapter 5, Theorem 5.2]{DSMN}).

Let $\sl_m$ be the $\Z_p$-Lie subalgebra of $\gl_m$ consisting of  matrices with  zero trace.
The algebra $\sl_m$ is also powerful, and then  $Sl_m^1:=exp(\sl_m)$ is uniform. More, since $\sl_m(\Q_p):=\Q_p \otimes \sl_m$ is simple, one has $\sl_m(\Q_p)=(\sl_m(\Q_p,\sl_m(\Q_p))$ which implies that  the abelianization of $Sl_m(\Z_p)$ is finite.
 Observe that  $exp \circ Trace =det  \circ exp$, confirming that $Sl_m^{1}=exp(\sl_m)$ is also the subgroup of $Gl_m^{1}$ of matrices of determinant~$1$.

\subsubsection{Uniform groups and $\Z_p$-Lie algebras}

For $k\geq 1$, let $\varphi_k$ be the reduction map: $$\varphi_k : Gl_m(\Z_p) \rightarrow Gl_m(\Z/p^k\Z).$$
Set $Gl_{m}^{(k)}=ker(\varphi_{k+\varepsilon})$ and $Sl_{m}^{(k)}=ker(\varphi_{k+\varepsilon})\cap Sl_m(\Z_p)$

\begin{prop}\label{lemm_filtration_sln} 

$(i)$ One has $Gl_m^1=Gl_{m}^{(1)}$ and $Sl_m^1=Sl_{m}^{(1)}$.

$(ii)$ The subgroups $Gl_{m}^{(k)}$ (resp. $Sl_{m}^{(k)}$) correspond to the $p$-descending central series of $Gl_m^{1}$ (resp. $Sl_m^{1}$). In other words, $Gl_m^{(k)}=(Gl_m)_k$ and $Sl_m^{(k)}=(Sl_m)_k$.

$(iii)$  For $k\geq 1$ one has $Gl_{m}^{(k)}=exp(p^{k-1}  \gl_m)$, and $Sl_{m}^{(k)}=exp(p^{k-1}  \sl_m)$.

\end{prop}

\begin{proof} For $(i)$ and $(ii)$ see  \cite[Chapter 5, Theorem 5.2]{DSMN}; for $(iii)$ see   \cite[Chapter 4, Lemma 4.14]{DSMN}.
\end{proof}

Proposition \ref{lemm_filtration_sln} is a  special case of the following result: 

\begin{theo}\label{theo_correspondence}
 There is a correspondence   between the category of uniform pro-$p$ groups $G$ and the category of powerful $\Z_p$-Lie algebras $\LL$. 
 When $G \subset Gl_m^{1}$ this correspondence is given by the exponential and the logarithm; in particular  $\LL=log(G) \in \gl_m$. 
\end{theo}

\begin{proof}
 See \cite[Chapter 9, Theorem 9.10]{DSMN}.
\end{proof}

\begin{defi}
 Let $G \subset Gl_m^{1}$ be a uniform pro-$p$ group of dimension $d$. Set $\Gg:=log(G) \subset \gl_m$, and $\Gg_p:=\Gg/p\Gg$. Observe that $\Gg_p$ is a $\fq_p$-vector space of dimension $d$.
\end{defi}

As for $Gl_m^1$ in   Proposition \ref{lemm_filtration_sln}, the $p$-descending central series $(G_n)$ of a uniform group $G \subset Gl_m(\Z_p)$  is easy to describe. Indeed:

\begin{prop}
 One has $G_n=exp(p^{n-1}\Gg)$. 
 
 In particular, $G_n/G_{n+1}\simeq p^{n-1}\Gg/p^n \Gg\simeq \Gg_p$.
\end{prop}

\begin{proof}See \cite[Chapter 4, Lemma 4.14]{DSMN}.
\end{proof}

   \subsubsection{The Lie algebra $\Gg$ as a sub-module of $\gl_m$} \label{section_liealgebra_uniform}
 
 Let $G \subset Gl_m^{1}$ be uniform; set $\Gg=log(G)$. Recall that $\Gg$ is the  powerful sub-Lie $\Z_p$-algebra of $ \gl_m$ such that $exp(\Gg)=G$. 
Let $\Delta'$ be a finite  subgroup of $Gl_m(\Z_p)$ of order coprime to $p$, acting 
  by conjugation  on~$G$; observe that  $\Delta'$ also acts on $Gl_m$, on $\gl_{m,p}:=\gl_m/p\gl_m$, and on $\Gg_p$. 
 Since $p\nmid |\Delta'|$, the  $\Z_p[\Delta']$-module  $\gl_m$ is projective    (see \cite[Chapter 14, \S 14.4]{Serre-representation}) and then, $\gl_{m,p}$  and $\gl_m(\Q_p):=\Q_p\otimes \gl_m$ have the 'same' character (as $\Delta'$-modules). 
 Of course, for the same reason, $\Gg_p$ and $\Gg(\Q_p)$ have the same character. Since  $\Gg(\Q_p) \subset \gl_m(\Q_p)$ we obtain:
 
 \begin{prop} \label{prop_compa_ad} Let $\Delta' \subset Gl_m(\Z_p) $ be a subgroup of order  coprime to $p$ acting on~$\Gg$ by conjugation.
  Then $\Gg_p$ is a sub-$\Delta'$-module of $\gl_{m,p}$.
 \end{prop}

 \begin{defi} When the action is given via a Galois representation $\rho_0:\Delta \rightarrow Gl_m(\Z_p)$ (here $\Delta'=\rho_0(\Delta)$),
  the $\Delta$-module $\Gg_p$ is called {\it the adjoint of~$G$ following $\rho_0$}.
 \end{defi}

\subsection{Semisimple  algebras}
The next Theorem, due to Kuranishi (\cite{Kuranishi}), is essential for our strategy.
See also \cite{Bois}. 

\begin{theo}[\cite{Kuranishi}] \label{theo_kuranishi}
 Let $\Ll$ be a semisimple $\Q_p$-Lie algebra. Then $\Ll$ can be  generated by $2$ elements.
\end{theo}

Let $\LL \subset \gl_m$ be a powerful $\Z_p$-Lie algebra.
For $x \in \LL$, put   $w_\LL(x):=max\{k, x \in p^k\LL\}$, $w_\LL(0)=\infty$; it is a valuation on~$\LL$ (following Lazard's terminology, see \cite[Chapter I, \S 2.2]{Lazard}).
When starting with a uniform group $G$, for $g\in G$ define $w_G(g):=w_\Gg(log(g))$, where $\Gg=log(G)$: this is a filtration on~$G$ (see \cite[Chapter II, \S 1]{Lazard}).

\begin{defi}
 Two topological groups $G$ and $G'$ are said to be  {\it locally the same} if they have a common open subgroup.
\end{defi}

As corollary of Theorem \ref{theo_kuranishi} we get

\begin{coro} \label{coro_locallyequal}
 Let $G\subset Gl_m^1$ be a uniform group such that $\Gg(\Q_p)$ is semisimple. Then there  exist two elements
 $g$ and $g'$ in $G$  such that
 \begin{enumerate}
  \item[$(i)$]
$w_G(g)=w_G(g')$,
\item[$(ii)$]  $g \notin \langle g'\rangle   G_{k+1}$, 
\item[$(iii)$]  the group $G$ and  the (closed) subgroup~$G'$ generated by  $g$ and $g'$,  are locally the same.
\end{enumerate}
\end{coro}

\begin{proof}
Let $\Gg:=log(G)$ be the powerful $\Z_p$-Lie algebra associated to $G$, and equipped with   the valuation $\omega_\Gg$. 
Set $\Ll:=\Q_p\otimes \Gg$. 
By Theorem \ref{theo_kuranishi} there exist $x,y \in \Ll$ such that $\Ll=\langle x,y \rangle$. 
By multiplying $x$ and $y$ by some powers of $p$, we can assume that $x$ and $y$ have the same valuation~$k$ (and are also in $\Gg$).
Suppose now that $x\equiv a_0 y \ {\rm mod \ } p^{k+1}\Gg$ for some $a_0 \in\Z_p \backslash p\Z_p$; then $x-a_0 y$ and $p^{k_1}y$  are of the same valuation $k_1+k$ for some $k_1\geq 1$. Suppose moreover that $x - a_0 y \equiv a_1p^{k_1}y \ {\rm mod \ } p^{k+k_1+1} \Gg$ for some $a_1 \in \Z_p \backslash p\Z_p$; then for some $k_2$, the elements $x-a_0y-a_1p^{k_1}y$ and $p^{k_2}y$ are of the same valuation $k_2+k\geq k_1+k+1$. If this process does not stop, we can construct a sequence of integers  $(k_n)$,  $k_{n+1}>k_n$,  and a sequence of $p$-adic integers $(a_n)$ such that $x-a_0y-a_1p^{k_1}y-\cdots -a_np^{k_n} y$ is of valuation $k_{n+1} +k$, showing that $x \in \langle y \rangle$, which is impossible since $\Ll$ is not abelian.
In conclusion, there exists $a_0,\cdots, a_{k_i} \in \Z_p  \backslash p\Z_p$, and integers $k_1,\cdots, k_i$ such that $x':=x-a_0y-\cdots a_{k_{i}}p^{k_{i}}y$ is of valuation $k+{k_{i+1}}$, but such that $ x' \notin  \langle p^{k_{i+1}} y\rangle + p^{k+k_{i+1}+1}\Gg$.

By abuse we note $x$ by $x'$, $p^{k_{i+1}}y$ by $y$, and $k+k_i$
by $k$. 
Thus, we may assume that $x$ and $y$ are in $\Gg$ with the  same valuation $k$, that they generate~$\Ll$, and that $\{ x,y \} $ is free in $p^k \Gg/p^{k+1}\Gg \simeq \Gg_p \simeq (\fq_p)^d$, where $d$ is the dimension of $G$. 

\smallskip

Set $g=exp(x)$ and $g'=exp(y)$. Then by the previous observations one has: $g \notin  \langle g'\rangle   G_{k+1}$. 
Let  $G'=\langle g,g'\rangle$ be the closed subgroup of $G$ generated by $g$ and $g'$.
The pro-$p$ group $G'$ is $p$-adic analytic as closed subgroup of a $p$-adic analytic group; let $U$ be an open uniform subgroup of $G'$. Then for $r\gg 0$, $g^{p^r}$ and $(g')^{p^r}$ are in $U$. Hence the $\Z_p$-Lie algebra $\Ll_U=log(U)$ of $U$ contains $p^rx$ and $p^ry$, and then $\Q_p\otimes \Ll_U=\Ll$. Thus, $U$ and $G$ are locally isomorphic and even locally the same (due to the fact that $U\subset G$),
see for example \cite[Part II, Chapter V, \S 2, Corollary 2]{Serre-Lie}, or \cite[Chapter 9, \S 9.5, Theorem 9.11]{DSMN}.  In other words, $G$ and $G'$ are locally the same.
\end{proof}

The two next examples make explicit Theorem \ref{theo_kuranishi}.

\begin{exem} \label{exemple_sl2}
Take $m=2$. Set $x=E_{1,2}(p)+E_{2,1}(p)$, and $y=E_{1,1}(p)-E_{2,2}(p)$. Observe that $(x,y)=2p\big(E_{2,1}(p)-E_{1,2}(p)\big)$, hence $x$ and $y$ generate the Lie algebra $\sl_2(\Q_p)$.
Set $g=exp(x)$ and $g'=exp(y)$, and $G'=\langle g,g'\rangle$. Then $G'$ has $Sl_{2}^{(2)}$ as open subgroup.
\end{exem}

\begin{exem}[\cite{Detinko-DeGraaf} or \cite{chisto-sln}] \label{example_Sln} Take $m\geq 3$.
 The Lie algebra $\sl_m$ is simple. Set $x= \sum_{i=1}^{m-1} E_{i,i+1}(p)$, and 
 $$y=\left\{ \begin{array}{ll} E_{m,1}(p) & m {\rm \ odd}, \\
E_{m-1,1}(p)+E_{m,2}(p)       & m {\rm \ even}.   
        \end{array}\right.$$
Observe that $\langle x,y\rangle_{\Z_p} \subset \sl_m$.  Thanks to \cite[Proposition 2.5 and Proposition 2.6]{Detinko-DeGraaf} and \cite[Example 2]{chisto-sln} one has $\langle x,y\rangle = \sl_m(\Q_p)$. 
Put $g=exp(x)$, $g'=exp(y)$ and $G'=\langle g,g' \rangle \subset Gl_m^1$. Observe that $w_G(g)=w_G(g')=1$. Then  $G'$ has $Sl_{m}^{(k)}$ as open subgroup for some $k\gg 0$.
\end{exem}

  \section{Lifting in uniform pro-$p$ groups} \label{section_main_embedding}
  To simplify we take $p>2$.
   The goal of this section is to give  lifting criteria for uniform groups including the well-known conditions when $G=Sl_m^1$ of $Gl_m^1$ (see \cite[\S 1.6]{Mazur}).

\subsection{Compatible actions} \label{section_compatible}

  Let $\GG$ be a pro-$p$ group of $p$-rank  $\geq d$, and  
   let $\Delta \subset Aut(\GG)$ be finite of order coprime to $p$.  Set $\Gamma= \GG \ltimes \Delta$.
   
   \medskip
   
   Let $\GG^{p,el}:=\GG/\GG^p[\GG,\GG]$ be the maximal abelian $p$-elementary
    quotient of $\GG$; observe that $\GG^{p,el}$  can be seen as a $\fq_p[\Delta]$-module.

    Let $M$ be a sub-$\fq_p[\Delta]$-module of $\GG^{p,el}$, and let $\rho_0: \Delta \rightarrow Gl_m(\Z_p)$ be a representation of~$\Delta$ such that $ker(\rho_0)$ acts trivially on $M$. Put $\Delta'=\rho_0(\Delta)$.
    Hence $M$ is also a $\Delta'$-module by
    $\rho_0(s)\cdot m:=s\cdot m$.
   
   \medskip
   
   Let $Pr_M: \GG \rightarrow \GG^{p,el} \rightarrow M$ be the projection of $\GG$ on $M$.
   
\medskip

Let $G'\subset Gl_m(\Z_p)$ be a pro-$p$ group such that $d_p G'=d_p M$.
  Suppose that $\rho_0(\Delta)$ acts on $G'$ by conjugation. Hence 
  $(G')^{p,el}$ becomes a $\Delta$-module via  $\rho_0$, by $s\cdot g':=\rho_0(s)\cdot g'$.
  We suppose now that the action of $\Delta$ on $M$ is compatible with that of $\Delta$ on $(G')^{p,el}$: in other words, 
  $\chi((G')^{p,el}) = \chi(M)$, as $\Delta$-module. Hence there exists one $\Delta$-isomorphism $ \beta: (G')^{p,el} \isomto M$ (which is equivalent to be  an isomorphism of $\Delta'$-modules).

  \subsection{Embedding problem} \label{section_embedding}

Let $G \subset Gl_m^1$ be a uniform pro-$p$ group of dimension~$d$. Set $\Gg:=log(G) \subset \gl_m$.
Given $1 \leq s\leq d$ and $k\geq 0$,  let $z_1,\cdots, z_s \in p^k\Gg$ be some free  elements  in  $p^{k}\Gg/p^{k+1}\Gg \simeq (\Z/p)^d$. 
Set $g_i=exp(z_i)$. Then for $i=1, \cdots, k$, one has $w_G(g_i)=k$.

\smallskip

Let us consider the closed subgroup $G'$ of $G$ generated by the $g_i$'s.
The group $G'$ is $p$-adic analytic. Observe that $G' \subset G_k \subset Gl_{m}^{(k)}=ker(Gl_m(\Z_p)\rightarrow Gl_m(\Z/p^k))$.
  Recall that $(G_n)$ is the $p$-central descending series of $G$.

  \medskip
  
  For $n\geq 1$, put $G'_{[n]}:=G'\cap G_{n+k-1}$. Hence $G'_{[1]}=G'$.
  
  \begin{lemm} \label{lemm_decalage}
$(i)$    The pro-$p$ group $G'$ is of $p$-rank $s$,  and  $(G')^{p,el} \simeq G'/G'_{[2]}$.

  $(ii)$ 
   For each $n\geq 1$, $G_{[n]}' \lhd G'$,   the quotient $G_{[n]}'/G_{[n+1]}'$ is $p$-elementary abelian, and  $G'$ acts trivially (by  conjugation) on $G_{[n]}'/G_{[n+1]}'$.

$(iii)$ 
  The  $G'_{[n]}$ are open in $G'$, and $\displaystyle{\bigcap_n G'_{[n]}=\{1\}}$. 
  \end{lemm}

  \begin{proof}
   $(i)$ One has the commutative diagram: 
   $$\xymatrix{G'/G'_{[2]} \ar@{^{(}->}[r]&G_k/G_{k+1} \ar@{->}^{\simeq}_{log}[r] &p^k\Gg/p^{k+1}\Gg \\
   & \ar@{->>}[lu]^-{P} G'/(G')^p[G',G']\ar@{->}_{log}[ru] & }$$
   Hence the family $\{g_1G'_{[2]},\cdots, g_sG'_{[2]}\}$ is free in $G'/G'_{[2]}$, showing that $d_p G' \geq d_p G'/G'_{[2]} \geq s$. But $G'$ is generated by the $g_i$'s. Thus $d_p G'=s$, and $P$ is an isomorphism.

   $(ii)$  Clearly $G_{[n]}' \lhd G'$.
   Since $G_{n+1}=G^p_n[G,G_n]$ one has:
   $$\begin{array}{rcl}
   G_{[n]}'/G_{[n+1]}'&=& G'\cap G_n/G'\cap G_{n+1} \\
   &=& \big(G'\cap G_n\big) G_n^p [G,G_n]\big/ G_n^p [G,G_n].
   \end{array}
   $$
   Hence $G_{[n]}'/G_{[n+1]}'$ is abelian, and  $G$ and then $G'$ acts trivially on $G_{[n]}'/G_{[n+1]}'$.

   $(iii)$ Point $(ii)$ shows that the $G_{[n]}$ are of finite index in $G'$, and then open since~$G'$ is pro-$p$  finitely generated. Regarding the intersection, that is obvious since $\displaystyle{\bigcap_n G_{n}=\{1\}}$. 
  \end{proof}

  We now include conditions  of Section \ref{section_compatible}.
  
 Via $\beta$ and $\rho_0$, suppose that $(G')^{p,el}$ can be seen as a sub-$\Delta$-module of $\GG^{p,el}$; or equivalently, $(G')^{p,el}$ is $\Delta'$-isomorphic to a subspace $M$ of $\GG^{p,el}$.

  Hence there exists a surjective morphism $f_2: \Gamma  \rightarrow G'/G_{[2]}'\ltimes \Delta'$, such that
  \begin{enumerate}
      \item[$(i)$] $(f_2)_{|\GG}=\beta^{-1}\circ Pr_M$,
\item[$(ii)$] $(f_2)_{|\Delta}=  \rho_0$. 
\end{enumerate}
Recall that $G'/G_{[2]}'=(G')^{p,el}$.
  
  \medskip
  
  More generally, suppose that for some  $n\geq 2$, there exists a
  surjective morphism $f_n : \Gamma  \rightarrow G'/G_{[n]}' \ltimes \Delta'$, where $(f_n)_{|\Delta}= \rho_0$.
 Then   let us consider  the embedding problem $(\E_{n})$:
  
  $$\xymatrix{ & & & \Gamma=\GG \ltimes \Delta \ar@{.>}[ld]_-{\psi_n} \ar@{->>}[d]^{f_n} \\
1 \ar[r] &G'_{[n]}/G'_{[n+1]}  \ar[r] & G'/G_{[n+1]}' \ltimes \Delta' \ar@{->>}[r]_{g_n} &  G'/G_{[n]}' \ltimes \Delta'  }$$
where   ${g_n}$ is the natural map (in particular ${g_n}_{|\Delta'}$ is the identity). 

 Thanks to the criteria of Hoechsmann (see for example \cite[Chapter III, \S 5]{NSW}),   $(\E_{n})$ has some solution when $H^2(\Gamma,G'_{[n]}/G'_{[n+1]})=0$, where the action of $\Gamma$ on $G'_{[n]}/G'_{[n+1]}$ is induced by conjugation via $f_n$. See for example \cite[Chapter III, \S 5, Proposition 3.5.9]{NSW}.
 In fact we need more:
 
 \begin{prop}\label{prop_strong}
  If $(\E_n)$ has a solution $\psi_n$, then $\psi_n$ is an epimorphism (the solution is called proper). 
 \end{prop}

 \begin{proof}
  The question is to see if the map $\psi_n$ is surjective. Since $G'/G'_{[n+1]}$ and $G'/G'_{[n]}$ are $p$-groups, it is equivalent to see if these two groups have the same minimal number of generators: that is Lemma \ref{lemm_decalage}, $(i)$.
 \end{proof}

 \subsection{Main Theorem}
 
 We can now announce the key theoretical result of our paper.
 Let us write $\GG^{ab} \simeq \T \oplus \Z_p^t$, where $\T$ is the torsion part of $\GG^{ab}$.
 Let us keep the notations of the previous sections. In particular $G$ is a uniform group of dimension $d$, $G'$ is a closed  subgroup of $G$, $\beta$ is a $\Delta$-isomorphism from $(G')^{p,el}$ to a sub-$\Delta$-module of $\GG^{p,el}$, $\rho_0:\Delta \rightarrow Gl_m(\Z_p)$ is a representation of $\Delta$, and  $\Delta'=\rho_0(\Delta)$.
 We suppose moreover that $\Delta'$ acts  by conjugation on $G$.  Hence, via $\rho_0$, the group $\Delta$ acts also on $\Gg:=log(G)\subset \gl_n$, and on $\Gg_p:=\Gg/p\Gg$ (see \S \ref{section_liealgebra_uniform}).

 \begin{theo} \label{theo_main} With the above notations, suppose given $f : \Gamma=\GG \ltimes \Delta \twoheadrightarrow G'/(G')^p[G',G']\ltimes \Delta'$ where $f_{|\Delta}= \rho_0$, such that: 
$(i)$ $H^2(\GG,\Q_p/\Z_p)=0$; and $(ii)$ $\T[p] \perp \Gg_p$.
 Then the embedding problem
   $$\xymatrix{   & \Gamma=\GG \ltimes \Delta \ar@{.>>}[ld]_{\psi} \ar@{->>}[d]^-{f} \\
  G'\ltimes \Delta'  \ar@{->>}[r]_-g &  G'/G'_{[2]}  \ltimes \Delta' }$$
 has a (proper) continuous solution $\psi$.
 \end{theo}

 \begin{proof} It is a proof step by step. 
 
 $\bullet$ First, for $n\geq 2$ suppose given  a
  surjective morphism $f_n : \Gamma  \rightarrow G'/G_{[n]}' \ltimes \Delta'$, where $(f_n)_{|\Delta}= \rho_0$. And consider the embedding problem $(\E_n)$.

 $\bullet$ Observe now that 
 $$
 \xymatrix{G'_{[n]}/G'_{[n+1]} \ar@{=}[r]& G'\cap G_n\big/G'\cap G_{n+1} \ar@{->}^-\sim[r] &  (G'\cap G_n)G_{n+1}/ G_{n+1}  \ar@{^(->}[d] \\
  & G_n/G_{n+1} &  G_n G_{n+1}/G_{n+1} \ar@{->}^-\sim[l] }$$
  Since $G$ is uniform, $G_n/G_{n+1} \simeq \Gg_p$, and this isomorphism is also compatible with the action of $\Delta$.
  In particular, $G'_{[n]}/G'_{[n+1]}$ is a sub-$\Delta$-module of $\Gg_p$.

 $\bullet$ Since $f_n(\GG)\subset G'/G'_{[n]}$, by Lemma \ref{lemm_decalage}  the group $\GG$ acts trivially (via $f_n$) on $ G_{[n]}'/G_{[n+1]}'$.
  By Theorem \ref{theo_H2_semisimple} we get  
    $$H^2(\Gamma, G_{[n]}'/G_{[n+1]}') \simeq \Big(\T[p]^\wedge\otimes  G_{[n]}'/G_{[n+1]}' \Big)^\Delta.$$
  
  $\bullet$ But by hypothesis 
$ \T[p] \perp \Gg_p$. Then as $G_{[n]}'/G_{[n+1]}' \hookrightarrow \Gg_p$, one has 
$ \T[p] \perp G_{[n]}'/G_{[n+1]}'$.
  By Lemma \ref{lemm_character_trivial}  we finally get  $H^2(\Gamma,G_{[n]}'/G_{[n+1]}')=0$:
  the embedding problem $(\E_n)$ has some proper solution $\psi_n$ thanks to  Proposition~\ref{prop_strong}.
  
  Put $f_{n+1}:=\psi_n$. 
  
  $\bullet$ By hypothesis $f_2$ is given. Hence by the previous computation one deduces that $(\E_2)$ has a proper solution, which  gives the existence of one $f_3$. Then $(\E_3)$ has a proper solution, etc.  
 To conclude, it suffices to take the projective limit of a system of compatible solutions $\psi_n$, and to remember  that $\displaystyle{\bigcap_n G'_{[n]}=\{1\}}$.
 \end{proof}

\begin{rema}\label{rema_lift}
Observe that $G'\ltimes \Delta' \hookrightarrow G_m(\Z_p)$. Hence the continuous map~$\psi$ induces a continuous Galois  representation $\rho : \Gamma \rightarrow Gl_m(\Z_p)$ with image containing $G'$ as open subgroup. Moreover for $\delta \in \Delta$, one has $\psi(\delta)=\rho_0(\delta)$; thus $\rho_{|\Delta} \simeq \rho_0$. In other words, $\rho$ is a lift of~$\rho_0$.
Finally observe that changing the map $\beta$ (which is possible since $p>2$), changes the representation $\rho$. 
 \end{rema}


\section{Applications}

Before developing the arithmetical context, let us make a quick observation. 

\begin{prop} Let $k$ be a number field such that $r_2 >0$.
 Suppose  Leopoldt and Gras conjectures for $k$ at $p$. Take $p\gg 0$. Then for every $p$-analytic group $G$ for which the Lie algebra  is semisimple, there exist continuous  Galois representations $\rho : Gal(\overline{k}/k) \rightarrow Gl_m(\Z_p)$ with image locally the same as~$G$.
\end{prop}

\begin{proof} Here we assume that  the pro-$p$ group $G_{k,p}$ is free of $p$-rank $r_2+1$.
Let $U \subset G$ be a uniform subgroup of $G$. The group $U$ is pro-$p$. We can assume that $U\subset Gl_m^1$,
 and we conclude with  Corollary \ref{coro_locallyequal} (as consequence of Theorem \ref{theo_kuranishi}).
\end{proof}

When $k$ is totally real,  one strategy is to start with a residual Galois representation of $Gal(\overline{k}/k)$ of order coprime to $p$ (typically of order $2$) in which at least one real place is ramified.

 \subsection{The principle}\label{section_Galois_lifting}
 
 We apply Section \ref{section_liealgebra_uniform} in our favorite arithmetical context 
as  developed by Greenberg \cite{Greenberg}, Ray \cite{Ray}, etc.

$\bullet$ Let us start with a Galois extension $K/k$ of Galois group $\Delta$ of order coprime to $p$.
Recall that $\Delta$ acts on $G_{K,p}$, etc.
Set $\Gamma=Gal(K_p/K) \simeq G_{K,p} \ltimes \Delta$.

Suppose  $ker(\iota_{K,p})$ trivial (equivalently, assume Leopoldt's conjecture for $K$ at $p$). Then  $H^2(G_{K,p},\Q_p/\Z_p)=0$ by Proposition \ref{prop_injectivity_iota}.

 $\bullet$ Let  $\rho_0 : \Delta \rightarrow Gl_m(\Z_p)$ be a Galois representation of $Gal(K/k)$.

For $i=1,\cdots,s$, let $L_i/K$ be cyclic degree $p$ extensions in $K_p/K$. Let $L$ be the compositum of the $L_i$'s and set $M=Gal(L/K)$.  We suppose that $\Delta$  acts on $M$  but also  that $ker(\rho_0)$ acts trivially on $M$ as in Section \ref{section_compatible}. Hence $\Delta':=\rho_0(\Delta)$ acts on $M$ by $\rho_0(s)\cdot m:= s \cdot m$.

$\bullet$ Let $G\subset Gl_m^1$ be a uniform  group, and let $G'$ be an open subgroup of $G$  as in Section \ref{section_embedding}.
Recall that $G'=\langle g_1,\cdots, g_s\rangle$ where the $g_i$'s are in $ G_{k}\backslash G_{k+1}$. In particular $G' \subset G_k$.
Observe  that $G_{k+1}=G^{p^{k+1}}$ by Theorem \ref{theo_uniform}.
Write $G/p^{k+1}:=(G \ {\rm mod \ } G^{p^{k+1}})$.

We suppose now that  $\rho_0(\Delta)$ acts by conjugation on $G'$, such that there exists a $\Delta$-isomorphism $\beta : (G')^{p,el} \rightarrow M$ (which is equivalent to say that is a $\Delta'$-isomorphism).

Hence, we also get $Gal(L/K) \ltimes \rho_0(\Delta) \simeq (G')^{p,el} \ltimes \Delta'$.
 By Lemma \ref{lemm_decalage} recall that $$(G')^{p,el}\simeq G'/(G')^p[G',G']\simeq G'/G'_{[2]} \simeq G'G_{k+1}/G_{k+1}.$$

  We then have a continuous Galois representation $$\rho_1: Gal(K/k) \rightarrow G/p^{k+1}\ltimes \Delta'$$ such that:
 \begin{enumerate}
  \item[$(i)$]  $(\rho_1)_{|Gal(K_p/K)}=\beta^{-1}\circ Pr_M$, 
  \item[$(ii)$]  ${\rho_1}_{|_{Gal(K/k)}}=\rho_0$, 
  \item[$(iii)$]$\rho_1 \ {\rm mod \ } G^{p^k}  \simeq \rho_0$.
 \end{enumerate}
 
 The  Galois representation $\rho_1$ plays the role of the  function $f$ of Theorem~\ref{theo_main}.

$\bullet$ As $\Delta'$ (or $\Delta$)  acts by conjugation on $G'$, we assume moreover that it also acts on $G$.
Set $\Gg:=log(G) \subset \gl_n$.   Hence $\Gg_p$ becomes a $\Delta$-module (via $\rho_0$).

\smallskip

 As consequence of Theorem \ref{theo_main} and Remark \ref{rema_lift}, we get:
 
 \begin{coro} \label{coro_sln}
  If $\ker(\iota_{K,p})=1$ and  $\T_{K,p}[p] \perp \Gg_p$, then the representation $\rho_0$ lifts to a Galois representation $\rho: Gal(K_p/k) \rightarrow Gl_m(\Z_p)$
  with image containing $G'$ as open subgroup.
 \end{coro}


  \subsection{Galois representations via imaginary quadratic fields} \label{section_imaginary}
  
  We start with an imaginary quadratic extension $K/\Q$. Let $p>2$ be a prime number. 
  Put $\Delta=Gal(K/\Q)=\langle s \rangle$, and let $\varphi$ be the nontrivial character of  $\Delta$.
  
  \smallskip
  
  $\bullet$ Suppose that $p\nmid |Cl_K|$. For $p=3$, we assume moreover that $\U_p/\iota_{K,p}(E_K)$ is torsion free; typically $K=\Q(\sqrt{-3})$.
The pro-$p$ group $G_{K,p}$ is free (see Example \ref{exemple_quad_imaginary}), and 
  $\chi(G_{K,p}^{ab}/p)=\1+\varphi$ by Proposition \ref{prop_character_Gp}. 
  Take $M=G_{K,p}^{p,el} =\langle h_1, h_2 \rangle \simeq (\Z/p)^2$, such that $s\cdot h_1=h_1$ and $s\cdot h_2=h_2^{-1}$.

  \smallskip
  
  $\bullet$ We recall observation of Example  \ref{example_Sln} from \cite{chisto-sln} and \cite{Detinko-DeGraaf}.  
  Take $m\geq 3$, and 
  consider $z_1=E_{1,2}(p)+E_{2,3}(p)+\cdots+E_{m-1,m}(p) \in \gl_m$, and 
  $$z_2=\left\{\begin{array}{ll} E_{m,1}(p) & m {\rm \ odd} \\ 
   E_{m-1,1}(p) +E_{m,2}(p) & m {\rm \ even}.
   \end{array}\right.$$
   Set $g_1=exp(z_1) \in Gl_m^1$ and $g_2=exp(z_2) \in Gl_m^1$, and $G'=\langle g_1,g_2 \rangle$. 
   Take the uniform group $G:=Sl_m^1$. Of course $G' \subset G$.
As seen in \ref{example_Sln} (thanks to Corollary \ref{coro_locallyequal}), the analytic groups  $G'$ and  $Sl_m(\Z_p)$ are locally the same.

  \smallskip
  
  $\bullet$ Set $A=\sum_{i}(-1)^{i+1}E_{i,i}$.
  By conjugation, $A\cdot z_1=-z_1$ and $A\cdot z_2=z_2$, and then $A$ acts by $-1$ on $g_1$ and by $+1$ on $g_2$. Of course $A$ acts also on $Sl_m(\Z_p)$.
  
  Let $\rho_0: Gal(K/\Q) \rightarrow Gl_m(\Z_p)$ be the Galois representation defined by $\rho_0(s)=A$. Here $ker(\rho_0)=1$, and the map $\beta: M \rightarrow (G')^{p,el}$ defined by $\beta(h_1)=g_1(G')^p[G',G']$ and $\beta(h_2)=g_2(G')^p[G',G']$ is an isomorphism of $\Delta$-modules.
  
  \smallskip
  
  For $m=2$, consider Example \ref{exemple_sl2} and take $z_1=E_{1,1}(p)-E_{2,2}(p)$, $z_2=E_{1,2}(p)+E_{2,1}(p)$, $g_1=exp(x_1)$, $g_2=exp(x_2)$, and $A=E_{1,1}-E_{2,2}$.

  \smallskip
  
 In conclusion, principle of Section \ref{section_Galois_lifting} allows us to lift $\rho_0$ to a Galois representation of $Gal(K_p/\Q) \rightarrow Gl_m(\Z_p)$. 
  
  \begin{theo} Given $p>3$, and $m\geq 1$. 
  Let $K/\Q$ be an imaginary quadratic extension such that 
  $p\nmid |Cl_K|$. Then there exist continuous Galois representations $\rho: Gal(K_{p}/\Q)\rightarrow Gl_m(\Z_p)$ with  open image.
  \end{theo}
  
  \begin{proof} Here the field $K$ is $p$-rational and $E_K=1$; then apply Corollary \ref{coro_sln}. Hence there exists a continuous Galois representation $\rho' : Gal(K_p/\Q) \rightarrow  Sl_m^1\ltimes \rho_0(\Delta) \hookrightarrow Gl_m(\Z_p)$ with image containing $Sl_m^k$ for some $k\gg 0$, as open subgroup.

  Let  $\omega' : G_\Q \rightarrow \Z_p^\times$ be  the cyclotomic character.
  Now, recall that since $Sl_m(\Q_p)$ is semisimple, every open subgroup of $Sl_m^1$ has finite abelianization.
  Hence  the image of the Galois representation ${\rho}:=\rho' \otimes \omega' : Gal(K_p/\Q) \rightarrow Gl_m(\Z_p)$ has $p$-adic dimension $m^2$; in conclusion the image of $\rho$ is open in $Gl_m(\Z_p)$.
  \end{proof}

  As corollary, we obtain:
  
  \begin{coro} \label{coro_maintheorem}
   There exist continuous Galois representations $\rho: Gal(\overline{\Q}/\Q) \rightarrow Gl_m(\Z_p)$ with open image satisfying:
   \begin{itemize}
       \item[$(i)$] $\rho$ is unramified ouside $\{p, \infty\}$ if $p\equiv -1 \ {\rm mod} \ 4$,
       \item[$(ii)$] $\rho$ is unramified ouside $\{2,p, \infty\}$ if $p\equiv 1 \ {\rm mod} \ 4$.
   \end{itemize}
  \end{coro}
  
  \begin{proof}
  Take $K=\Q(\sqrt{-p})$. 
  Thanks to an explicit version of Brauer-Siegel (see for example \cite{Louboutin}), $p\nmid |Cl_K|$.
  (For $p=3$, the number field $\Q(\sqrt{-3})$ is $3$-rational).
  \end{proof}

\subsection{Galois representations via $K=\Q(\zeta_p)$} \label{section_zeta_p}

The study of Galois representations through $\Q(\zeta_p)$  allows us to realize for large $m$, Galois representations $\rho:G_\Q \rightarrow Gl_m(\Z_p)$ unramified outside $\{p, \infty\}$, and with open image.

\medskip

Take $k=\Q$, $K=\Q(\zeta_p)$. 
Let $s$ be a generator of $\Delta=Gal(K/\Q)$. 
Remind that $\iota_{K,p}$ is injective, and  by Proposition \ref{prop_character_Gp},  $\chi(\ff_{K,p}/p)=\1+\omega+\omega^3+\cdots + \omega^{p-2}$,
where $\omega: G_\Q \rightarrow \fq_p^\times \subset \Z_p^\times$ is the mod $p$ reduction of the cyclotomic character.

\smallskip

 Take $m\geq 3$. Let   $g_1$ and $g_2$ be the elements of $Sl_m^1$ as in the previous section. Set $G'=\langle g_1,g_2\rangle \subset Sl_m^1$.

\smallskip

Given an odd integer $a$,  set $\displaystyle{A_a(s)=\sum_{i=1}^m \omega^{ia}(s)E_{i,i}}$. Consider 
 the  Galois representation ${\rho_0}: Gal(K/\Q) \rightarrow Gl_m(\Z_p)$ defined by ${\rho_0}(s)=A_a(s)$.

Then $A_a(s)\cdot z_1=\omega^{-a}(s) \ z_1$ and
$$A_a(s) \cdot z_2=\left\{\begin{array}{ll} \omega^{a(m-1)}(s) \ z_2 & m {\rm \ odd} \\ 
  \omega^{a(m-2)}(s)\ z_2  & m {\rm \ even}.
   \end{array}\right.$$
   Put $g_1=exp(z_1)$ and $g_2=exp(z_2)$.
The action of $A_a(s)$  is odd on $g_1$, and even on $g_2$.
Of course $A_a(s)$ acts also on $Sl_m^1$.

\smallskip
 
Thanks to  the decomposition of $\chi(\ff_{K,p}/p)$, we can find $h_1$ and $h_2$ in $\ff_{K,p}$ such that $s\cdot h_1= h_1^{\omega^{_a}(s)} $, and $s\cdot h_2=h_2^{\omega^{a(m-1)}(s)} $  if $a(m-1)=0 \ {\rm mod \ }p-1$ for $m$ odd, and $s\cdot h_2=h_2^{\omega^{a(m-2)}}$ if $a(m-2)=0 \ {\rm mod} \ p-1$ for $m$  even; there is no condition for the odd character, but the even character must be trivial.

Put $M=\fq_p h_1 + \fq_p h_2 \subset G_{K,p}^{p,el}$. Then $\Delta$ acts on $M$,  $ker(\rho_0)=ker(\omega^a)$ acts trivially on $M$, and the two $\Delta$-modules $M$ and $(G')^{p,el}$ are isomorphic.

Here, it is not difficult to see that the character $\chi(\gl_{m,p})$ of $\gl_{m,p}$ (via  $\rho_0$) contains only characters like $\omega^{(i-j)a}$ with $i,j \in \{1,\cdots, m\}$.

\smallskip

We can apply the previous techniques. And,
as before, the representation $\rho_0$ lifts when $\omega^{a a'}$ does not appear in $\chi(\T_{K,p}[p])=\chi^*(Cl_K[p])$, for every $a'\in \{\pm 1, \pm 2, \cdots, \pm m\}$ (in fact class modulo $p-1$ of).

\smallskip

Take now $a$ the odd part of $p-1$; in other words, $p-1=a 2^\lambda$ with $2\nmid a$; so $\lambda=v_2(p-1)$.
We obtain the first condition (regarding the existence of $h_1$ and $h_2$): for $m$ odd we must have  $v_2(m-1) \geq v_2(p-1)$; for $m$ even we must have $v_2(m-2) \geq v_2(p-1)$.
For a regular prime $p$, that is the only condition.

Regarding the condition so that $\T_{K,p}\perp \gl_{m,p}$: Let us start with a character $\omega^{k_i}$ that appears in $\chi(Cl_K[p])$, that is equivalent to say that $\omega^{1-k_i}$ appears in $\chi(\T_{K,p}[p]$); if  $\omega^{1-k_i}$ appears in $\chi(\gl_{m,p})$ then
$a $ divides  $k_i-1$.

\smallskip

Let us look at quickly the $p\equiv 3$ mod $4$ case; here   $a=(p-1)/2$. 

There is no condition on $m$, and the condition regarding $\T_{K,p}$  becomes $r_{\omega^{(p-1)/2}}(\T_{K,p})=0$. Observe that    $r_{\omega^{(p-1)/2}}(\T_{K,p})=r_\varphi(\T_{K_0,p})$, where $K_0=\Q(\sqrt{-p})$ and where $\varphi$ is the nontrivial character of $Gal(K_0/\Q)$. Hence  since $r_\varphi \T_{K_0,p}=0$ (see the proof of Corollary \ref{coro_maintheorem}), we get that  there is no obstruction for the embedding problem. In fact,  observe that in this case the representation we obtain through $\Q(\zeta_p)$ can be deduced by the one of Corollary \ref{coro_maintheorem}.

\smallskip

We have proved:

\begin{theo}  \label{theo_zeta_p} 
Let $p \equiv 1 \ {\rm mod \ 4 }$ be a prime number, and let $m\geq 3$.
Write $p-1=2^\lambda a$ where $2\nmid a$. 
Let $\{\omega^{k_1},\cdots, \omega^{k_e}\}$ be the characters corresponding to the nontrivial components of  the $p$-Sylow of the  class group of $\Q(\zeta_p)$.
Suppose that:

\begin{itemize}
 \item[$(i)$] $v_2(m-1) \geq v_2(p-1)$ if $m$ is odd, and  $v_2(m-2) \geq v_2(p-1)$ if $m$ is even;
  \item[$(ii)$] $a\nmid (k_i-1)$ for $i=1,\cdots, e$.
\end{itemize} Then there exist continuous Galois representations $\rho: G_\Q \rightarrow Gl_m(\Z_p)$   unramified outside $\{p,\infty\}$, and with open image.
\end{theo}

\begin{coro}
Let $p \equiv 1 \ ({\rm mod \ 4 })$ be a regular prime.  Then there exist continuous  Galois representations $\rho:G_\Q \rightarrow Gl_m(\Z_p)$ unramified outside $\{p, \infty\}$ and with open image, for every  odd $m$ such that $v_2(m-1)\geq v_2(p-1)$, and for every even $m$ such that $v_2(m-2)\geq v_2(p-1)$.
\end{coro}


\end{document}